\documentclass[secthm,seceqn,amsthm,ussrhead,12pt]{amsart}
\usepackage{amsmath,latexsym}
\usepackage[english]{babel}
\usepackage[psamsfonts]{amssymb}
\usepackage{times}
\usepackage{cite}

\usepackage[mathcal]{euscript}
\numberwithin{equation}{section} \textwidth=17.5cm
\newtheorem{thm}{Theorem}[section]
\newtheorem{lem}[thm]{Lemma}
\newtheorem{exa}[thm]{Example}

\newtheorem{rem}[thm]{Remark}

\numberwithin{equation}{section}

\begin{document}


\baselineskip=15pt



\title{Geometric characterization  of $L_1$-spaces}

\author[N. Yadgorov]{Normuxammad Yadgorov}
\address{National University of Uzbekistan\\
Vuzgorodok,  100174, Tashkent,    Uzbekistan}
\email{yadgorovm@mail.ru}

\author[M.  Ibragimov]{Mukhtar Ibragimov}
\address{Karakalpak state university\\
  230113 Nukus,    Uzbekistan}
\email{mukhtar\_nukus@mail.ru}

\author[K.  Kudaybergenov]{Karimbergen  Kudaybergenov}
\address{Karakalpak state university\\
  230113 Nukus,    Uzbekistan}
\email{karim2006@mail.ru}

\date{}

\begin{abstract}
The paper is devoted to a description of all strongly facially
symmetric spaces which are isometrically isomorphic to
$L_1$-spaces. We prove that if  $Z$ is a real neutral strongly
facially symmetric space such that  every  maximal geometric
tripotent from the dual space of $Z$ is unitary  then, the space
$Z$ is isometrically isomorphic to the space
  $L_1(\Omega, \Sigma,
\mu),$ where
  $(\Omega, \Sigma,
\mu)$  is an appropriate  measure space having the direct sum  property.
 \end{abstract}

\maketitle

\section{Introduction}

One of the main problem in operator algebras is a geometric
characterization of operator algebras and operator spaces.  In
this connection
 in  papers of Y.~Friedman and B.~Russo the
 so-called facially symmetric spaces were  introduced
(see \cite{fr1, fr2, fr3, fr4, fr5,  MFAT, NR}). In \cite{fr5},
the complete structure of atomic facially symmetric spaces was
determined. More precisely, it was shown that an irreducible,
neutral, strongly facially symmetric space is linearly isometric
to the predual of one of the Cartan factors of types $1$  to $6,$
 provided that it satisfies some natural
and physically significant axioms, four in number, which are known
to hold in the preduals of all $JBW^\ast$-triples.

The project of classifying facially symmetric spaces was started
in \cite{fr4}, where, using two of the pure state properties,
denoted by $STP$ and $FE,$ geometric characterizations of complex
Hilbert spaces and complex spin factors were given. The former is
precisely a rank $1$ $JBW^\ast$-triple and a special case of a
Cartan factor of type $1,$
 and the latter is the Cartan factor of type $4$
  and a special case of a $JBW^\ast$-triple of rank $2.$
   The explicit structure of a spin factor naturally
embedded in a facially symmetric space was then used in \cite{fr5}
to construct abstract generating sets and complete the
classification in the atomic case. In \cite{NR}   a geometric
characterization of
 the dual ball of global  $JB^\ast$-triples was given.

The present paper is devoted to a description of all real strongly
facially symmetric spaces which are  isometrically isomorphic to
$L_1$-spaces. Using Kakutani's characterization of real
$L_1$-spaces, we show that a neutral strongly facially symmetric
space in which every maximal geometric tripotent is unitary, is
isometrically isomorphic to an $L_1$-space. None of the extra
axioms used in \cite{fr4, fr5,NR} are assumed.

\section{Facially symmetric spaces}

In this section we shall recall some basic facts and notation
about facially symmetric spaces (see for details \cite{fr1, fr2,
fr3, fr4, fr5}).

Let $Z$  be a real or complex normed space. Elements $x, y \in Z$
 are \textit{orthogonal}, notation $x \diamondsuit y,$
  if $\|x + y\|=\|x - y\|=\|x\|+\|y\|.$
Subsets   $S, T\subset Z$ are said to be
 \textit{orthogonal}, notation   $(S \diamondsuit T),$ if   $x
\diamondsuit y$ for all   $(x, y)\in S\times T.$
 A~\textit{norm exposed face} of the unit ball
 $Z_1$  of $Z$  is a non-empty set (necessarily $\neq Z_1$)
  of the form
$F=F_u=\{x\in Z: u(x)=1\},$
   where   $u\in Z^\ast,$   $\|u\| =1.$
 Recall that a face $G$  of a convex set $K$  is a non-empty
convex subset of $K$  such that if $\lambda y+(1-\lambda)z\in G,$
where
 $y, z\in K,$ $\lambda
\in (0,1)$ implies    $y, z\in G.$
 In particular, an extreme point of $K$
  is a face of $K.$
   An element $u\in Z^\ast$
    is called  \textit{ a projective unit}
     if $||u||=1$  and
     $\langle u, y\rangle=0$ for all   $y\in F_u^\diamondsuit.$
      Here, for any subset $S,$
      $S^\diamondsuit$
      denotes the set of all elements orthogonal to each element of
      $S.$

A  norm exposed face $F_u$ in $Z_1$ is said to be
\textit{symmetric face} if
 there is a linear isometric symmetry  $S_u$
  of $Z$ onto $Z,$ with $S_u^2
= I$ such that the fixed point set of $S_u$ is
$(\overline{\mbox{sp}}F_{u})\oplus F_{u}^\diamondsuit.$

Recall that a  normed space $Z$
 is said to be \textit{weakly facially symmetric} (WFS) if every norm exposed face in
$Z_1$  is symmetric.

For each symmetric face $F_u$ the contractive projections
$P_k(F_u),$ $k = 0, 1, 2$ on $Z$ defined as follows. First
$P_1(F_u) = (I-S_u)/2$
 is the projection on the $-1$  eigenspace of $S_u.$
  Next  define $P_2(F_u)$
and $P_0(F_u)$
 as the projections of $Z$
  onto $\overline{sp}F_u$ and
    $F_u^\diamondsuit,$  respectively, so that
    $P_2(F_u) +
P_0(F_u) = (I + S_u )/2.$ A \textit{geometric tripotent} is a
projective unit $u$
 with the property that $F_u$
  is a symmetric face and
$S_u^\ast u=u$  for  symmetry $S_u$  corresponding to $u.$
 The projections $P_k(F_u)$
  are called geometric
Peirce projections.

$\mathcal{GT}$  and $\mathcal{SF}$
 denote the collections of geometric tripotents and symmetric faces respectively, and the
map $\mathcal{GT} \ni u \mapsto F_u \in \mathcal{SF}$
 is a bijection \cite[Proposition 1.6]{fr2}.
 For each geometric tripotent $u$ in the dual
of a WFS space $Z,$
 we shall denote the geometric Peirce projections by
 $P_k(u) = P_k(F_u), k = 0, 1, 2.$
  Two elements $f$
 and $g$ of $Z^\ast$ are orthogonal if one of them belongs to
 $P_2(u)^\ast(Z^\ast)$ and the other to $P_0(u)^\ast(Z^\ast)$
for some geometric tripotent $u.$

A contractive projection $Q$
 on a normed space $Z$  is said to be \textit{neutral}
  if for each $x \in Z,$
   $\|Q(x) \|= \|x\|$ implies
$Q(x)=x.$  A normed space $Z$
 is \textit{neutral} if for every symmetric face $F_u,$
  the projection $P_2(F_u)$ is neutral.

 A WFS space $Z$
 is \textit{strongly facially symmetric}
  (SFS) if for every norm exposed face $F_u$ in $Z_1$
   and every
$g \in Z^\ast$ with  $\|g\| = 1$ and  $F_u \subset F_g,$ we have
$S_u^\ast g = g,$ where  $S_u$ denotes a symmetry associated with
$F_u.$

The principal examples of   neutral complex strongly facially
symmetric spaces are preduals of complex $JBW^\ast$-triples, in
particular, the preduals of von Neumann algebras, see \cite{fr3}.
 In these cases, as shown in \cite{fr3}, geometric
tripotents correspond to tripotents in a $JBW^\ast$-triple and to
partial isometries in a von Neumann algebra.

In a neutral strongly facially symmetric space $Z,$
 every non-zero element has a polar decomposition
 \cite[Theorem 4.3]{fr2}:
  for nonzero $x\in Z$
  there exists a unique geometric tripotent $v = v_x$
   with $\langle v, x\rangle  = \|x\|$ and
$\langle v, x^\diamondsuit \rangle = 0.$ If $x, y\in Z,$ then
$x\diamondsuit y$ if and only if $v_x \diamondsuit v_y,$  as
follows from \cite[Corollary 1.3(b) and Lemma 2.1]{fr1}.

A partial ordering can be defined on the set of geometric
tripotents as follows: if $u, v \in \mathcal{GT},$ then $u \leq
v,$ if   $F_u \subset F_v,$ or equivalently (\cite[Lemma
4.2]{fr2}) $P_2(u)^*v = u$ or   $v - u$  is either zero  or a
geometric tripotent orthogonal to   ê $u.$

\section{Main result}

Henceforth   ''face'' means ''norm exposed face''.

Let   $Z$ be a real neutral strongly facially symmetric space. A
geometric tripotent  $u\in \mathcal{GT}$ is said to be

--  \textit{maximal}
 if  $P_0(u)=0;$

--  \textit{unitary} if $P_2(u)=I.$

It is clear that any unitary  geometric tripotent is  maximal.

Notice that a geometric tripotent $e$ is a unitary if and only if
the convex hull of the set $F_e\cup F_{-e}$ coincides with the
unit ball $Z_1,$ i.e.
\begin{align}\label{yad}
Z_1=\mbox{co}\{F_e\cup F_{-e}\}.
\end{align}
Also note that property \eqref{yad}  is a much stronger property
than the Jordan decomposition property of a face (see
\cite[Lemmata 2.3-2.6]{NR}). Recall that a face $F_u$ satisfies
the  Jordan decomposition property if its
 real span coincides with the geometric Peirce $2$-space of the geometric
 tripotent $u.$

\begin{exa} The space  $\mathbb{R}^n$
with the   norm
$$
||x||=\sum\limits_{i=1}^{n}|t_i|,\, x=(t_i)\in \mathbb{R}^n
$$
is a  SFS-space. If  $e\in \mathbb{R}^n\cong (\mathbb{R}^n)^\ast$
is a maximal geometric tripotent then
$$
e=(\varepsilon_1, \varepsilon_2, \ldots, \varepsilon_n),
  \varepsilon_i\in\{-1, 1\}, i\in\overline{1, n},
$$
and in this case the face
$$
F_e=\left\{x\in \mathbb{R}^n:
 \sum\limits_{i=1}^{n}\varepsilon_i  t_i=1,\, \varepsilon_i
 t_i\geq 0, i=\overline{1, n}\right\}
$$
satisfies   \eqref{yad}.

More generally, consider a measure space  $(\Omega, \Sigma, \mu)$
with measure $\mu$  having the direct sum property, i.e. there is
a family
 $\{\Omega_{i}\}_{i\in
J}\subset\Sigma,$ $0<\mu(\Omega_{i})<\infty,\,i\in J,$ such that
for any $A\in\Sigma,\,\mu(A)<\infty,$ there exist a countable
subset $J_{0 }\subset J$ and a set  $B$ with zero measure such
that $A=\bigcup\limits_{i\in J_{0}}(A\cap \Omega_{i})\cup B.$

  Let  $L_1(\Omega, \Sigma, \mu)$ be the  space of all real integrable functions
on $(\Omega, \Sigma, \mu).$ The space $L_1(\Omega, \Sigma, \mu)$
with the   norm
$$
||f||=\int\limits_\Omega|f(t)|\,d\mu(t),\, f\in L_1(\Omega,
\Sigma, \mu)
$$
is a  SFS-space. If  $e\in L^\infty(\Omega, \Sigma, \mu)\cong
L_1(\Omega, \Sigma, \mu)^\ast$ is a maximal geometric tripotent
then
$$
e=\tilde{\chi}_A-\tilde{\chi}_{\Omega\setminus A} \quad\mbox{for
some} \quad  A\in \Sigma,
$$
where  $\tilde{\chi}_A$ is the  class containing  the indicator
function of the set $A\in \Sigma.$
 Then the face
$$
F_e=\left\{f\in L_1(\Omega, \Sigma, \mu):
||f||=1,\,\int\limits_\Omega e(t)f(t)\,d\mu(t)=1\right\}
$$
satisfies   \eqref{yad}.
\end{exa}

The next result is the main result of the paper,  giving a
description of all strongly facially symmetric spaces which are
isometrically isomorphic to   $L_1$-spaces.

\begin{thm}\label{MTH}
Let  $Z$ be a real neutral strongly facially symmetric space such
that every maximal geometric tripotent from   $Z^\ast$ is unitary.
Then there exits
 a measure space  $(\Omega, \Sigma,
\mu)$ with measure $\mu$  having the direct sum  property such
that the space  $Z$ is  isometrically isomorphic to the space
  $L_1(\Omega, \Sigma,
\mu).$
\end{thm}

For the proof we need several lemmata.

Let  $u, v\in \mathcal{GT}.$
 If  $F_u\cap F_v\neq \emptyset$
then by  $u\wedge v$ we denote the   unique geometric tripotent
such that
  $F_{u\wedge v}=F_u\cap F_v,$
  otherwise we set
$u\wedge v =0.$

\begin{lem}\label{cap}
Let  $e$ be   unitary   and let  $v \in \mathcal{GT}.$ Then
$F_v\cap F_e\neq\emptyset$ or  $F_{-v}\cap F_e\neq\emptyset.$
 \end{lem}

\begin{proof}
Let  $x\in  F_v.$ By  equality  \eqref{yad} we obtain that
$$
x=ty+(1-t)z
$$
for some
 $y,\, -z\in F_e$ and $0\leq t \leq1.$

If  $t=1$ or $t=0$ then  $x=y$ or $x=z,$ respectively.  Hence
$x\in F_{v}\cap F_e$ or $-x\in F_{-v}\cap F_e.$

Let $0<t<1.$ Since $F_v$ is a  face,  $y, z\in F_v.$ Therefore
$F_v\cap F_e\neq\emptyset$ and  $F_{-v}\cap F_e\neq\emptyset.$ The
proof is complete.
\end{proof}

 \begin{lem}\label{JDM}
 Let  $e$ be  unitary. Then for every  $u \in \mathcal{GT}$
  there exist  mutually orthogonal
geometric tripotents $u_1, u_2\leq e$ such that  $u=u_1-u_2.$
\end{lem}

\begin{proof}
Put
\begin{gather*}
u_1=u\wedge e, \,  u_2=(-u)\wedge e.
\end{gather*}

Let us prove that
\begin{gather*}
u_1\diamondsuit u_2,\, u=u_1-u_2.
\end{gather*}
Let  $x_1\in F_{u_1}$ and $x_2\in F_{u_2}.$ Then
$$
x_1, x_2 \in F_e,\,\, x_1, -x_2\in F_{u},
$$
and therefore
$$
\frac{x_1+x_2}{2}\in F_e, \frac{x_1-x_2}{2}\in F_{u}.
$$
Thus
$$
\left\|\frac{x_1+x_2}{2}\right\|=1,
\left\|\frac{x_1-x_2}{2}\right\|=1,
$$
and
$$
||x_1+x_2||=||x_1-x_2||=2=||x_1||+||x_2||.
$$
Hence   $x_1\diamondsuit x_2,$ and therefore   $u_1\diamondsuit
u_2.$

Now suppose that  $v=u-u_1+u_2\neq 0.$ By Lemma  \ref{cap} we know
that  that    $F_v\cap F_e\neq\emptyset$ or  $F_{-v}\cap
F_e\neq\emptyset.$   Without loss of generality it can be assumed
that $F_v\cap F_e\neq\emptyset.$ Thus there exists an element
$x\in Z_1$ such that
$$
\langle v, x\rangle=\langle e, x\rangle=1.
$$
  Since
$v\leq u,$ we have   $\langle u, x\rangle=1.$ Thus  $x\in F_u\cap
F_e,$ i.e. $x\in F_{u_1}$ or  $\langle u_1, x\rangle=1.$ Since
$u_1\diamondsuit u_2,$ we have
 $\langle u_2, x\rangle=0.$ Hence
$$
\langle v, x\rangle= \langle u, x\rangle- \langle u_1, x\rangle+
\langle u_2, x\rangle =0,
$$
 a contradiction. The proof is complete.
\end{proof}

\begin{lem}\label{max}
Let  $u, w$ be orthogonal geometric tripotents. Then $u+w$ is
maximal if and only if $u-w$ is  maximal.
\end{lem}

\begin{proof}
Let $u+w$ be   maximal. Suppose that $u-w$ is not maximal. Then
there exists a maximal geometric tripotent
 $e$ such that
$e> u-w.$ Set $w_1=e-u+w.$ Then $w_1\diamondsuit u$ and
$w_1\diamondsuit w.$ Therefore $u+w<u+w+w_1.$ This contradicts the
 maximality of $u+w.$
 The
proof is complete.
\end{proof}

Recall   that a  face $F$ of a convex set $K$ is called
\textit{split face} if there exists a face $G$ ($F\cap
G=\emptyset$), called complementary to $F,$ such that $K$ is the
direct convex sum $F\oplus_{c}G;$ i.e. any element $x\in K$ can be
uniquely represented in the form $x=t y+ (1-t)z,$ where $t\in [0,
1],$ $y\in F,$ $z\in G$ (see e.g. \cite[P.~420]{AS}, \cite{AY}).

\begin{lem}\label{decom}
Let  $u, w$ be orthogonal geometric tripotents. If $u+w$ is
maximal then
\begin{align}\label{deco}
F_{u+w}= F_u\oplus_{c} F_w.
\end{align}
\end{lem}

\begin{proof}
First we shall  show that
$$
F_{u+w}=\mbox{co}\{F_u\cup F_w\}.
$$
It  suffices to show that
$$
F_{u+w}\subseteq  \mbox{co}\{F_u\cup F_w\}.
$$
By Lemma~\ref{max} the geometric tripotent $u-w$ is  maximal.
Therefore the face $F_{u-w}$ satisfies equality \eqref{yad}, i.e.
$$
Z_1=\mbox{co}\{F_{u-w}\cup F_{w-u}\}.
$$
Thus every
 element $x\in F_{u+w}$ has the form
 $$
x=ty+(1-t)z
$$
for some
 $y, -z \in F_{u-w}$ and $0\leq t \leq1.$

Consider the following three cases.

Case 1.  If $t=0$ then $x\in F_{u+w}\cap F_{w-u}=F_w.$

Case 2.  If $t=1$ then $x\in F_{u+w}\cap F_{u-w}=F_u.$

Case 3. If $0<t<1$ then applying  the geometric tripotent $u+w$ to
the equality $x=ty+(1-t)z$ we obtain
\begin{align}\label{equal}
t u(y)+t w(y)+(1-t) u(z)+ (1-t) w(z)=1.
\end{align}
Since
 $y\in F_{u-w}$ and $z\in F_{w-u}$ we see  that
 $$
u(y)-w(y)=1,\, w(z)-u(z)=1.
 $$
Thus
\begin{align}\label{equa}
t u(y)-t w(y)-(1-t) u(z)+ (1-t) w(z)=1.
\end{align}
Summing \eqref{equal} and \eqref{equa} we get
$$
t u(y)+(1-t) w(z)=1.
$$
Since $|u(y)|\leq 1$ and $|w(z)|\leq 1$ the last equality implies
that
$$
u(y)=w(z)=1.
$$
This means that
 $y\in F_{u}$ and $z\in F_{w}.$
Therefore
 $$
x=ty+(1-t)z\in \mbox{co}\{F_u\cup F_w\}.
$$
Consequently $F_{u+w}=\mbox{co}\{F_u\cup F_w\}.$  Taking into
account that $F_u\diamondsuit  F_w$ we get  $F_{u+w}=
F_u\oplus_{c} F_w.$ The proof is complete.
\end{proof}

Let  $u$ be an arbitrary geometric tripotent and let $e$ be a
maximal geometric tripotent  such that $u\leq e.$ First we shall
show that
$$
Z=\overline{\mbox{sp}} F_u \oplus \overline{\mbox{sp}} F_w,
$$
where $w=e-u.$ Using equalities \eqref{yad} and \eqref{deco} we
obtain
\begin{align*}
 Z & = \mbox{sp} Z_1=\mbox{sp}\{\mbox{co}\{F_{e}\cup F_{-e}\}\}=\\
& = \mbox{sp} F_{e}=\mbox{sp}\{F_u\oplus_{c}  F_w\}=\mbox{sp}
F_u\oplus \mbox{sp} F_w,
\end{align*} i.e.
$$
Z=\mbox{sp} F_u \oplus \mbox{sp} F_w.
$$
From    $\mbox{sp} F_u\diamondsuit  \mbox{sp} F_w$ it follows that
$\overline{\mbox{sp}} F_u\diamondsuit  \overline{\mbox{sp}} F_w,$
and therefore
$$
Z=\overline{\mbox{sp}} F_u \oplus \overline{\mbox{sp}} F_w.
$$
This implies that
$$ P_2(u)+P_2(w)=I.
$$
Since  $P_1(u)P_0(u)=0$ and $P_2(w)=P_0(u)P_2(w)$ (see
\cite[Corollary 3.4]{fr2})
 we obtain $P_1(u)P_2(w)=0.$ Therefore
\[
P_1(u)=P_1(u)I=P_1(u)[P_2(u)+P_2(w)]=0.
\]
 So we have

\begin{lem}\label{P=0}
For every    $u\in \mathcal{GT}$ the projection   $P_1(u)=0$ is
zero.
\end{lem}

For orthogonal geometric tripotents $v_1, v_2$ we have
\begin{align}\label{ptwo}
P_2(v_1+v_2)=P_2(v_1)+P_2(v_2).
\end{align}

Indeed, by \cite[Lemma 1.8]{fr2} we have
\[
P_0(v_1+v_2)=P_0(v_1)P_0(v_2).
\]
Using the last equality and taking into account the equalities
$P_1(v_1)=P_1(v_2)=P_1(v_1+v_2)=0,$ together with  Corollary 3.4
of \cite{fr2}, we get
\begin{align*}
 P_2(v_1+v_2) & =I-P_0(v_1+v_2)=I^2- P_0(v_1+v_2)=\\
 & =(P_2(v_1)+P_0(v_1))(P_2(v_2)+P_0(v_2))- P_0(v_1)P_0(v_2)=\\
& = P_2(v_1)+P_2(v_2)+P_0(v_1)P_0(v_2)-P_0(v_1)P_0(v_2)=\\
& = P_2(v_1)+P_2(v_2).
\end{align*}

 Now we fix a unitary  $e\in \mathcal{GT}.$

On the space  $Z$ we define order (depending on $e$) by the
following rule:
\begin{align}\label{order}
x\geq y\, \Leftrightarrow\, x-y\in \mathbb{R}^+F_e.
\end{align}

\begin{lem}
\label{orsp} $Z$  is a partially ordered  linear space,  i.e.
\begin{enumerate}
\item $x\leq x;$

\item $x\leq y,\,  y\leq    z \Rightarrow  x\leq z;$

\item $x\leq y,\,  y\leq    x \Rightarrow  x=y;$

\item $x\leq y\,   \Rightarrow  x+z\leq y+z;$

\item $x\geq 0,\,  \lambda \geq 0 \Rightarrow  \lambda x\geq 0.$
\end{enumerate}
\end{lem}

\begin{proof}
The properties  (1), (4) and (5) are  trivial.

To prove (2), let  $x\leq y$ and $y\leq z.$ Then $y-x, \, z-y\in
\mathbb{R}^+F_e.$
 Thus  $z-x\in \mathbb{R}^+F_e,$
i.e. $x\leq z.$

For (3), let   $x\leq y,\,  y\leq    x.$ Then  $y-x=\alpha a$ and
$x-y=\beta b$ for some   $\alpha, \beta\geq 0$ and
  $a, b\in F_e.$
Therefore $\alpha a+ \beta b=0.$ Applying to this equality the
geometric tripotent  $e$ we obtain  $\alpha +\beta=0.$ Thus
$\alpha = \beta =0,$ i.e. $x=y.$ The proof is complete.
\end{proof}

\begin{rem}
Note that if  $v\leq e$ then  \cite[Lemma 2.4]{NR} implies that
\begin{align}\label{neal}
  P_k(v)(F_e)\subseteq F_e,\, k=0, 2.
\end{align}
\end{rem}

\begin{lem}
\label{uniq} Let
 $a, b, x, y\geq 0$ with
$a\diamondsuit  b.$  If $a-b= x-y$ then
$$
x-a=y-b\geq 0;
$$
in addition, if  $x\diamondsuit  y$ then  $x=a$ and $y=b.$
 \end{lem}

\begin{proof}
Let $v_a$ be  the  smallest geometric tripotent such that
 $v_a(a)=||a||$ (polar decomposition). Since $a\geq 0$ it follows that $v_a\leq e.$
 Applying  the projection $P_2(v_a)$ to the equality
$a-b=x-y$ we obtain
$$
P_2(v_a)(x)-P_2(v_a)(y)= P_2(v_a)(a-b)=P_2(v_a)(a)=a.
$$
Using  \eqref{neal} we get
$$
P_2(v_a)(x)-a=P_2(v_a)(y)\in \mathbb{R}^+F_e,
$$
and therefore
\begin{align*}
x-a & =P_2(v_a)(x)+P_0(v_a)(x)-a=\\
& =[P_2(v_a)(x)-a]+P_0(v_a)(x)\in \mathbb{R}^+F_e,
\end{align*}
i.e. $x\geq a.$

Now suppose that
  $x\diamondsuit  y.$
Then as shown above,  $x\geq a$ and $a\geq x.$ Thus  $x=a$ and
$y=b.$ The proof is complete.
\end{proof}

\begin{lem}
\label{nor} For  $x\in Z$ the following conditions
 are equivalent:
\begin{enumerate}
\item $x\geq 0;$

\item $||x||=\langle e, x\rangle.$
\end{enumerate}
\end{lem}

\begin{proof}
Take  $x\geq 0,$ i.e.  $x=\alpha y$ for some $\alpha\geq 0$ and
$y\in F_e.$ Then
$$
||x||=||\alpha y||=\alpha ||y||=\alpha =\alpha \langle e, y\rangle
= \langle e, x\rangle.
$$

Conversely, if    $||x||=\langle e, x\rangle, x\neq 0,$ then
$\frac{\textstyle x}{\textstyle ||x||}
 \in F_e,$ i.e. $x\geq 0.$
The proof is complete.
 \end{proof}

\begin{lem}
\label{uniqe} Every  element  $x\in Z$ can be  uniquely
represented as a sum
 $$
 x=x_+-x_-,
 $$
where
 $x_+, x_- \geq 0$
and  $x_+\diamondsuit  x_-.$
\end{lem}

\begin{proof}
Take the  smallest geometric tripotent $v_x\in \mathcal{GT}$ such
that
 $v_x(x)=||x||.$
By Lemma  \ref{JDM} there exist  mutually orthogonal geometric
tripotents
  $v_1, v_2\leq e$
such that $v_x=v_1-v_2.$ Put
$$
x_+=P_2(v_1)(x),\, x_-= - P_2(v_2)(x).
$$
By the proof of \cite[Theorem 4.3 (d)]{fr2} we get $\langle v_1,
x\rangle=||P_2(v_1)(x)||,$ and therefore
\begin{align*}
\langle e, x_+\rangle & =\langle e, P_2(v_1)(x)\rangle = \langle
P_2^\ast(v_1)e, x\rangle=\\
& = \langle v_1, x\rangle=||P_2(v_1)(x)||=||x_+||.
\end{align*}
  This means
that  $x_+\geq 0.$ Similarly $x_-\geq 0.$ Further using equality
\eqref{ptwo} we find  that $x=x_+-x_-$ and
 $x_+\diamondsuit x_-.$
Uniqueness follows  from Lemma  \ref{uniq}. The proof is complete.
\end{proof}

\begin{lem}
\label{latt}
  $Z$ is a lattice, i.e. for any   $x, y\in Z$ there exist
$$
x\vee y,\, x\wedge y\in Z.
$$
\end{lem}

\begin{proof}
By Lemma \ref{uniqe} there exist  mutually orthogonal elements
 $a, b \geq 0$ such that
 $x-y=a-b.$
Then
\begin{align}\label{mx}
x\vee y=\frac{x+y+a+b}{2},
\end{align}
\begin{align}\label{min}
x\wedge y=\frac{x+y-a-b}{2}.
\end{align}
Indeed,
\begin{align*}
 x\vee y-x & =\frac{x+y+a+b}{2}-x= \\
& =\frac{y-x+a+b}{2}=b\geq 0
\end{align*}
 and
\begin{align*}
x\vee y-y & =\frac{x+y+a+b}{2}-y= \\
 & =\frac{x-y+a+b}{2}=a\geq 0.
\end{align*}

Now let  $x, y\leq z,$ where $z\in Z.$ Denote
$$
x_1=z-x\geq 0,\, y_1=z-y\geq 0.
$$
Thus $x-y=y_1-x_1.$ Therefore $y_1-x_1=a-b.$ Lemma  \ref{uniq}
implies that
$$
y_1-a=x_1-b\geq 0.
$$
Further
$$
z-x\vee y=\frac{x+y+x_1+y_1}{2}-\frac{x+y+a+b}{2}=
$$
$$
=\frac{x_1+y_1-a-b}{2}=y_1-a\geq 0.
$$
This means that
$$
x\vee y=\frac{x+y+a+b}{2}.
$$
In the same way we can prove  equality   \eqref{min}. The proof is
complete.
\end{proof}

A Banach lattice   $X$ is said to be \textit{abstract  $L$-space}
if
$$
||x+y||=||x||+||y||
$$
for all   $x, y\in X$ with $x\wedge y =0$  (see  \cite[p. 14]{Lin}
and \cite{kak}).

\begin{lem}
\label{main}
  $Z$ is an abstract   $L$-space.
\end{lem}

\begin{proof} First we show that
\begin{align*}
0\leq x\leq y \Rightarrow ||x||\leq ||y||;
\end{align*}
\begin{align*}
||x||=||\,|x|\,||,
\end{align*} where $|x|=x_++x_-$ is the
absolute value of $x.$

Let $0\leq x\leq y.$ Then
\begin{align*}
||x||=\langle e, x\rangle\leq \langle e, y\rangle=||y||,
\end{align*} i.e.
\begin{align*}
||x||\leq ||y||.
\end{align*}
Further
\begin{align*}
|||x||| & =||x_++x_-||=[x_+\,\diamondsuit \,x_-]=\\
& =||x_+-x_-||=||x||.
\end{align*}
Hence $Z$ is a Banach lattice.

For    $x, y\geq 0,$ using Lemma  \ref{nor} we obtain
\begin{align*}
||x+y||=\langle e, x+y\rangle= \langle e, x\rangle+\langle e,
y\rangle= ||x||+||y||.
\end{align*}
 This means that  $Z$ is an abstract
$L$-space. The proof is complete.
\end{proof}

Now  Theorem  \ref{MTH} follows from Lemma \ref{main} and
\cite[Theorem  1.b.2.]{Lin}.

\begin{rem}
The following observations were kindly suggested by the referee,
to whom the authors are deeply indebted.

Theorem~\ref{MTH} fails for complex spaces. Indeed, by
\cite[Theorem 2.11]{fr3} for any finite von Neumann algebra its
predual  is a neutral strongly facially symmetric space in which
every maximal geometric  tripotent is unitary. However, that
predual is not isometric to an $L_1$-space, for example for the
algebra $B(H)$ of all bounded linear operators on the finite
dimensional  Hilbert space $H$ of dimension at least  $2.$

The predual of a real $JBW^\ast$-triple is a neutral weakly
facially symmetric space   (see \cite[Theorem~5.5]{ER} and
\cite[Theorem 3.1]{fr3}) which is not strongly facially symmetric.
The strong
  facial symmetry of the predual of a complex von Neumann algebra
  depends on the field being complex (see  the proof of Corollary 2.9
in \cite{fr3}). Indeed, if the predual of a non commutative real
von Neumann algebra were a strongly facially symmetric space, this
would contradict Theorem~\ref{MTH} above.
\end{rem}

\subsection*{Acknowledgments}

The authors would like to thank the referee for valuable comments
and suggestions.

\end{document}